\documentclass[11pt]{amsart}

\usepackage{amsmath,amssymb}
\usepackage[T1]{fontenc}
\usepackage[a4paper,top=2.5cm,bottom=2.5cm,inner=3.1cm,outer=3.1cm]{geometry}
\usepackage[hidelinks]{hyperref}

\newcommand{\IP}{\operatorname{IP}}
\newcommand{\cov}{\operatorname{cov}}
\newcommand{\VC}{\operatorname{VC}}

\theoremstyle{plain}
\newtheorem{theorem}{Theorem}[section]
\newtheorem*{theorem*}{Theorem}
\newtheorem{prop}[theorem]{Proposition}
\newtheorem{lemma}[theorem]{Lemma}

\theoremstyle{definition}
\newtheorem{defn}[theorem]{Definition}

\theoremstyle{remark}
\newtheorem{remark}[theorem]{Remark}

\numberwithin{equation}{section}

\title[The Composition Lemma for $n$-dependence]
{The Composition Lemma for $n$-dependence}
\author{Artem Chernikov and Yuyan He}
\date{}

\begin{document}

\begin{abstract}
We prove that a relation obtained by composing arbitrary functions of arity $\leq k$ with a relation definable in an $n$-dependent structure is $kn$-dependent. This confirms a conjecture of Chernikov and Hempel. We also demonstrate optimality of the result.
\end{abstract}

\maketitle

\section{Introduction}
The hierarchy of \(n\)-dependent theories, introduced by Shelah \cite{MR3273451, MR3666349},
extends the class of NIP theories: \(1\)-dependence is NIP, while
\(n\)-dependence excludes the uniform coding of arbitrary subsets of
finite \(n\)-dimensional boxes (see Definition \ref{def:ip}).   Basic properties of $n$-dependent theories are investigated in \cite{CPT}, where in particular the numerical parameter \emph{$\VC_n$-dimension} generalizing the classical $\VC$-dimension and characterizing $n$-dependence is defined and investigated quantitatively \cite[Proposition 3.9]{CPT}. Since then, $n$-dependent theories were studied further in pure model theory \cite{CPT, hempel2016n, chernikov2019mekler, chernikov2021n, CHIII}; the study of VC$_n$-dimension, or \emph{higher VC-theory}, found deep connections to (arithmetic) hypergraph combinatorics \cite{chernikov2020hypergraph, terry2021irregular, terry2023improved, terry2024growth, chernikov2024perfect, terry2025structure, terry2025quadratic, gishboliner2025regularity, sheats2025linear} and higher arity PAC learning  (PAC$_n$ learning) in product spaces \cite{Kobayashi, KotaPAC, chernikov2020hypergraph, chernikov2025higher, CorMal2025}.

The $n$-dependence property of formulas is known to be preserved under certain operations, e.g.~Boolean or continuous operations (\cite[Proposition 6.5]{CPT} and \cite[Proposition 10.5]{chernikov2020hypergraph}) and integration averaging, equivalently Keisler randomization (\cite[Section 10.2]{chernikov2020hypergraph}, generalizing \cite{zbMATH05662727} in the case $n=1$). Here we demonstrate that a relation obtained by composing \emph{arbitrary}  functions of arity $\leq k$ with a relation definable in an $n$-dependent structure remains $kn$-dependent:
\begin{theorem*}[Composition Lemma] 
Let \(n,k\geq1\) be arbitrary, and let
\(\mathcal M\) be an \(\mathcal L'\)-structure so that its reduct \(\mathcal M \restriction \mathcal L\) to a sublanguage \(\mathcal L\subseteq\mathcal L'\) is $n$-dependent.  Let \(I\) be finite
and let \(\rho((x_\alpha)_{\alpha\in I})\) be an
\(\mathcal L\)-formula.  For each \(\alpha\in I\), let $
 S_\alpha\subseteq\{0,\ldots, kn\}$  with \(|S_\alpha|\leq k\) 
and let \(t_\alpha((y_s)_{s\in S_\alpha})\) be an \(\mathcal L'\)-definable map whose values have the
same sorts as \(x_\alpha\).
Let 
\[
 \psi(y_0;y_1,\ldots,y_{kn})
 :=
 \rho\bigl(
   (t_\alpha((y_s)_{s\in S_\alpha}))_{\alpha\in I}
 \bigr).
\]
Then the \(\mathcal L'\)-formula  \(\psi\) is $kn$-dependent in $\mathcal{M}$.
\end{theorem*}

This Composition Lemma was introduced in \cite{chernikov2021n}, and established there in the special case $n=1, k=2$ using an infinitary type-counting criterion and set-theoretic absoluteness \cite[Theorem 5.12]{chernikov2021n}.
Following this, the case $n=1$ and $k$ arbitrary was proved in \cite[Theorem~3.24]{CHIII} using a finitary type-counting criterion and an array shattering lemma generalizing Sauer-Shelah. The general case, confirmed by our theorem here, was conjectured in \cite[Conjecture~3.27]{CHIII}. The original motivation in  \cite{chernikov2021n} was to establish \(2\)-dependence of
infinite-dimensional non-degenerate bilinear spaces over NIP fields; this was generalized to  \(n\)-dependence of non-degenerate 
alternating \(n\)-linear spaces over NIP fields in \cite{CHIII}. Further applications include $n$-dependence of generic nilpotent groups and Lie algebras over finite fields \cite{d2025model}; generic nilpotent Lie algebras over algebraically closed fields in a two-sorted language \cite{d2024two}; and pseudofinite quadratic geometries \cite{kestner2024some}. An analogue of the Composition Lemma for NFOP$_2$, with $k=2$ and the $\mathcal{L}$-reduct of $\mathcal{M}$ assumed stable, was proved in \cite{abd2023higher}. An analogous result for the stronger property NOP$_2$ (in the sense of
Takeuchi) was included in the preliminary version of \cite{CHIII}, and will be published in future work.

In fact, we obtain the Composition Lemma in a sharper form depending on the sets of
variables used by the maps.  Namely, let 
\(V=\{0,\ldots,d\}\), and suppose $
 \psi(y_0;y_1,\ldots,y_d)
 =
 \rho\bigl(
   (t_\alpha((y_s)_{s\in S_\alpha}))_{\alpha\in I}
 \bigr)$, 
where \(S_\alpha\subseteq V\).  For the family
\(\mathcal S=\{S_\alpha:S_\alpha\ne\varnothing\}\), define its \emph{covering number}   $
c
 := 
 \min_{\mathcal F\subseteq\mathcal S}
 \left(
   |\mathcal F|+
   \left|V\setminus\bigcup\mathcal F\right|
 \right)$. 
Theorem~\ref{thm:composition} shows that, if \(c\geq2\) and \(\psi\)
has \(\IP_d\), then the theory of the \(\mathcal L\)-reduct has
\(\IP_{c-1}\).  If every \(S_\alpha\) has size at most \(k\), then $
 c\geq\left\lceil\frac{d+1}{k}\right\rceil$, so 
taking \(d := nk\) we obtain that every such composition over an
\(n\)-dependent reduct is \(nk\)-dependent.  Our proof replaces the type-counting and higher-dimensional
shattering parts of the arguments used in \cite{chernikov2021n,CHIII} by an analysis of   indiscernible 
representations of the generic ordered partite hypergraph (Definition \ref{def: array rep}).  An
\(\IP_d\)-pattern for \(\psi\) yields a representation whose support
family is \(\mathcal S\).  The key Proposition~\ref{prop:array-bound} shows
that a representation with support family \(\mathcal S\) forces
\(\IP_{\cov_V(\mathcal S)-1}\).  Its proof is an induction relying on the following two operations on indiscernible representations.  For a
nonsingleton support \(E\), if the array obtained by omitting its
\(E\)-component is not order-indiscernible,
Lemma~\ref{lem:remove-support} produces a representation of the generic hypergraph omitting
\(E\) without decreasing the covering number (this is a refinement of the characterization of $n$-dependence via indiscernible collapse \cite[Theorem 5.4]{CPT}, which generalizes the NIP case  \cite{scow2012characterization}).  Otherwise,
Lemma~\ref{lem:compression} compresses all coordinates in \(E\) to
one new coordinate, again without decreasing the covering number (the automorphism argument is based on the proof in \cite[Section 5]{chernikov2021n}, but replaces type-counting with a direct shattering extraction).

In Section \ref{sec: Optimality} we demonstrate optimality of our result. In Proposition \ref{prop:local-hypotheses-sharp} we show that the global assumption that $\mathcal{M}\restriction \mathcal{L}$ is necessary, i.e.~simply assuming that the formula  $ \rho$ is $n$-dependent under arbitrary partitions of its variables is insufficient, already in the NIP case $n=1$ --- this gives a
negative answer to \cite[Problem~3.26(2)]{CHIII}. We also observe that  both the general
covering-number bound in Proposition \ref{prop:array-bound} and the uniform bound \(nk\) in Theorem~\ref{thm:composition} are optimal
(Propositions~\ref{prop:array-sharp} and
\ref{prop:composition-sharp}).

\subsection*{Acknowledgements}
The second author thanks Nick Ramsey for helpful discussions.
The first author was partially supported by the NSF
Research Grants DMS-2246598, DMS-2554164 and by the Alexander von Humboldt Foundation.
\subsection*{AI Statement}

All results are due to the authors. ChatGPT 5.6 was used to help with organizing the presentation and in preparation of the manuscript.

\section{$n$-dependence and representations of the generic hypergraph}

Fix a complete theory \(T\) and work in a monster model
\(\mathbb M\models T\).  For \(m\geq1\), write
\([m]=\{1,\ldots,m\}\).  For sets \(X\) and \(Q\), write \(X^Q\) for
the set of maps \(Q\to X\).  If
\(\bar u=(u_q)_{q\in Q}\in X^Q\), then
\(\bar u|_S=(u_q)_{q\in S}\) for \(S\subseteq Q\).

\begin{defn}\label{def:ip}
For \(r\geq1\), a formula
\(\varphi(x;y_1,\ldots,y_r)\) has \(\IP_r\) if, for every
\(m\), there are rows $(b_i^j:i\in[m])$ for $j\in[r]$ 
such that for every \(A\subseteq[m]^r\) some \(e_A\) satisfies
\[
 \varphi(e_A;b_{i_1}^1,\ldots,b_{i_r}^r)
 \quad\Longleftrightarrow\quad
 (i_1,\ldots,i_r)\in A.
\]
A theory is \(r\)-dependent if none of its formulas has \(\IP_r\).
\end{defn}

\begin{remark}\label{rem:basic}
\(\IP_s\) implies \(\IP_r\) for \(s\geq r\), by fixing one parameter
in each of \(s-r\) rows.  A formula with parameters having \(\IP_r\)
also gives a parameter-free formula with \(\IP_r\), by moving the
parameters into the object-variable block. Note also that naming or forgetting parameters does not affect $r$-dependence of $T$ (at the price of changing the formula and the witnessing rows). \end{remark}

\subsection{The index structure}
Let
\(\mathcal K_V\) be the class of finite \(V\)-partite structures
whose parts \(P_q\) may be empty, with a linear order \(<_q\) on each
part and an arbitrary relation $
R\subseteq\prod_{q\in V}P_q$. 
Let \(\mathcal G_V\) be the Fra\"{\i}ss\'e limit of
\(\mathcal K_V\). Denote the language by $\mathcal{L}_{V}$, and let $\mathcal{L}_{V}^{-} := \mathcal{L}_{V} \setminus \{R\}$. We let $
 P_S := \prod_{q\in S}P_q^{\mathcal G_V}$ for $S\subseteq V$.
 (We view \(\mathcal G_V\) as a multi-sorted structure with sorts $P_q$. In the equivalent one-sorted coding of the ordered $|V|$-partite $|V|$-uniform hypergraph in \cite{CPT}, the parts were named by predicates and their
orders were concatenated into a single order on \(V\).)

\begin{lemma}[One-point extension]\label{lem:one-point}
Let \(A\) be a finite set of vertices of \(\mathcal G_V\), let
\(q\in V\), and let \(J\) be a nonempty interval of \(P_q\)
determined by a cut over \(A\cap P_q\).  For every map
\[
 \epsilon:
 \prod_{v\ne q}(A\cap P_v)\longrightarrow\{0,1\}
\]
there is \(p\in J\setminus A\) such that
the following holds.  Given
\((a_v)_{v\ne q}\in\prod_{v\ne q}(A\cap P_v)\), define \(b_q=p\)
and \(b_v=a_v\) for \(v\ne q\).  Then
\[
 R((b_v)_{v\in V})
 \quad\Longleftrightarrow\quad
 \epsilon((a_v)_{v\ne q})=1.
\]
\end{lemma}

\begin{proof}
Adjoin \(p\) in the prescribed cut and prescribe the indicated
\(R\)-incidences.  The resulting finite structure belongs to
\(\mathcal K_V\), so the Fra\"{\i}ss\'e extension property embeds it
over \(A\) into \(\mathcal G_V\).
\end{proof}

\subsection{Arrays and representations}

Let \(\mathcal S\) be a finite family of nonempty subsets of \(V\).
An \emph{\(\mathcal S\)-array} is a family \(a=(a_S)_{S\in\mathcal S}\),
where each \(a_S\) maps \(P_S\) to tuples of fixed arity and sorts
(depending on \(S\)) in \(\mathbb M\).  For tuples
\(\bar b,\bar b'\) of the same sorts in  \(\mathbb M\),
write \(\bar b\equiv\bar b'\) if they have the same complete type. Given two tuples $\bigl(\bar p_\ell\bigr)_{\ell<L}$ and $\bigl(\bar p'_\ell\bigr)_{\ell<L}$ with   $\bar p_\ell=(p_{\ell,q})_{q\in S_\ell}$ and $\bar p'_\ell=(p'_{\ell,q})_{q\in S_\ell}$ in $P_{S_{\ell}}$, and $S_\ell \in \mathcal S$, they are \emph{index-isomorphic} if the coordinate assignment $p_{\ell,q}\longmapsto p'_{\ell,q}$ 
is a well-defined isomorphism between the substructures of
\(\mathcal G_V\) induced on $\{p_{\ell,q}\}_{\ell < L, q \in S_{\ell}}$ and $\{p'_{\ell,q}\}_{\ell < L, q \in S_{\ell}}$.  They are
\emph{order-isomorphic} if the same assignment is a well-defined
isomorphism in the reduct forgetting \(R\). (We note that the tuples $\bar p_\ell$ may have common elements as $\ell$ varies.)

\begin{defn}\label{def:indiscernible}
An \(\mathcal S\)-array \(a\) is \emph{\(V\)-indiscernible} if $
 (a_{S_\ell}(\bar p_\ell))_{\ell<L}
 \equiv
 (a_{S_\ell}(\bar p'_\ell))_{\ell<L}$ 
whenever the tuples $\bigl(\bar p_\ell\bigr)_{\ell<L}$ and $\bigl(\bar p'_\ell\bigr)_{\ell<L}$  are index-isomorphic.  It is
\emph{order-indiscernible} if the same implication holds whenever
they are order-isomorphic.
\end{defn}

\begin{defn}\label{def: array rep}
	Let \(\rho((z_S)_{S\in\mathcal S}) \in \mathcal{L}(\mathbb{M})\) be a formula, where \(z_S\)
has the same arity and sorts as the values of \(a_S\).  We say that 
\(a\) \emph{represents \(R\) via \(\rho\)} if
\begin{equation}\label{eq:represents}
 \mathbb M\models
 \rho\bigl((a_S(\bar p|_S))_{S\in\mathcal S}\bigr)
 \quad\Longleftrightarrow\quad
 \mathcal G_V\models R(\bar p)
 \qquad \textrm{ for all } \bar p\in P_V.
\end{equation}
A \emph{representation of \(R\)} is such a pair \((a,\rho)\), and
\(\mathcal S\) is its \emph{support family}.
\end{defn}

\begin{lemma}\label{lem:indiscernibilization}
Given a representation \((a,\rho)\) of \(R\), we can find a \(V\)-indiscernible  one (over the parameters of $\rho$) with
the same support family and representing formula.
\end{lemma}

\begin{proof}
Form a new $\mathcal{L}'$-structure $N$ by adding to $\mathbb{M}$ a new sort \(U\) containing distinct elements
\(u_p\) for \(p\in\mathcal G_V\); and,  for each \(S\in\mathcal S\), add a
tuple \(h_S\) of function symbols and interpret $
 h_S((u_{p_q})_{q\in S}) := a_S((p_q)_{q\in S})$ for $(p_q)_{q\in S}\in P_S)$, 
and define these functions arbitrarily on all other tuples.
By \cite[Definition~4.6 and Corollary~4.8]{CPT}, a saturated elementary
extension $\bar{N} \succ N$ contains elements \(c_p \in U(\bar{N}) \), \(p\in\mathcal G_V\), such that
\((c_{p_i})_{i<m}\equiv(c_{p'_i})_{i<m}\)
whenever $(p_i)_{i<m}$ and $(p'_i)_{i<m}$ have the same quantifier free type in $\mathcal G_V$.  Moreover, for every formula
\(\chi((w_i)_{i<m}) \in \mathcal{L}'\) and vertices $p_i \in \mathcal G_V$, there are vertices \(r_i \in \mathcal G_V\)
such that  $(p_i)_{i<m}$ and $(r_i)_{i<m}$ have the same quantifier free type in $\mathcal G_V$ and $\chi((c_{p_i})_{i<m}) 
  \Leftrightarrow 
 \chi((u_{r_i})_{i<m})$. 
Working in $\bar{N}$, set
\[
 a'_S((p_q)_{q\in S}) := h_S((c_{p_q})_{q\in S}).
\]
If \(p_{\ell,q}\mapsto p'_{\ell,q}\) is an isomorphism of the induced substructures of $\mathcal G_V$, the first
property gives
\(\bigl((c_{p_{\ell,q}})_{q\in S_\ell}\bigr)_{\ell<L}
\equiv
\bigl((c_{p'_{\ell,q}})_{q\in S_\ell}\bigr)_{\ell<L}\).
Applying \(h_{S_\ell}\) to the \(\ell\)-th subtuple gives
\((a'_{S_\ell}(\bar p_\ell))_{\ell<L}
\equiv
(a'_{S_\ell}(\bar p'_\ell))_{\ell<L}\).
Thus \(a'\) is \(V\)-indiscernible. For \(\bar p\in P_V\), apply the second property to $\theta(\bar w)=\rho((h_S(\bar w|_S))_{S\in\mathcal S})$. 
There is \(\bar r\in P_V\) such that \(p_q\mapsto r_q\) is an
isomorphism between the induced substructures of $\mathcal G_V$ and $
	\bar{N} \models \rho((a'_S(\bar p|_S))_{S\in\mathcal S}) \Leftrightarrow
\bar{N} \models  \rho((a_S(\bar r|_S))_{S\in\mathcal S})  \Leftrightarrow
\mathbb{M} \models  \rho((a_S(\bar r|_S))_{S\in\mathcal S}) 
  \Leftrightarrow R(\bar r)
 \Leftrightarrow R(\bar p)$. 
Hence \((a',\rho)\) is a representation of $R$ in $\bar{N}|_{\mathcal{L}} \succ \mathbb{M}$, so by saturation we find the required representation  in $\mathbb{M}$.
\end{proof}

\section{The array representation bound}

\begin{lemma}\label{lem:cone}
For each \(q\in V\), let \(W_q\subseteq P_q\) be finite and 
\(c_q\in W_q\).  Put $
 F_q=W_q\setminus\{c_q\}$, $
 W=\bigcup_{q\in V}W_q$, 
and let \(A\) be the substructure of $\mathcal{G}_{V}$ induced on \(W\).  Then there are countable
sets \(I_q\subseteq P_q\setminus W_q\) satisfying the following:  
\begin{enumerate}
\item every element of \(I_q\) has the same cut over \(F_q\) as
      \(c_q\) (in the order on $P_q$);
\item the structure induced on \(\bigcup_qI_q\) is isomorphic to
      \(\mathcal G_V\);
\item whenever \(x_q\in F_q\cup I_q\) for every \(q\), and at least
      one \(x_q\) lies in \(F_q\),
      \[
       R((x_q)_{q\in V})
       \quad\Longleftrightarrow\quad
       R^A((\kappa_q(x_q))_{q\in V}),
      \]
      where  \(\kappa_q(x)=x\) when \(x\in F_q\), and
\(\kappa_q(x)=c_q\) when \(x\in I_q\).

\end{enumerate}
\end{lemma}

\begin{proof}
Take disjoint countable sets \(I_q\) whose union carries a copy of
\(\mathcal G_V\).  In each part, place \(I_q\) above \(c_q\) within
the cut of \(c_q\) over \(F_q\).  On
\[
\prod_{q\in V}(F_q\cup I_q)\setminus
\prod_{q\in V}I_q
\]
define \(R\) by the equivalence in item~(3), and define it
arbitrarily elsewhere, retaining its given values on \(A\).

This is a countable structure all of whose finite substructures lie
in \(\mathcal K_V\).  Enumerating its new vertices and repeatedly
using Lemma~\ref{lem:one-point} embeds it into
\(\mathcal G_V\) over \(A\).  The images of the \(I_q\)'s have the
required properties.
\end{proof}

\begin{defn}
	 For a family \(\mathcal S\) of nonempty subsets of a
finite set \(V\), put
\[
 \cov_V(\mathcal S)
 =
 \min_{\mathcal F\subseteq\mathcal S}
 \left(
   |\mathcal F|+
   \left|V\setminus\bigcup\mathcal F\right|
 \right).
\]
Thus \(\cov_V(\mathcal S)\) is the least number of members of
\(\mathcal S\), supplemented by singletons, needed to cover \(V\).
\end{defn}

The following lemma is a refinement of the characterization of $n$-dependence via indiscernible collapse from \cite[Theorem 5.4]{CPT} (which generalizes \cite{scow2012characterization} in the NIP case).
\begin{lemma}[Removing a support]\label{lem:remove-support}
Let \(a\) be a \(V\)-indiscernible \(\mathcal S\)-array and let
\(E\in\mathcal S\).  If
\((a_S)_{S\in\mathcal S\setminus\{E\}}\) is not
order-indiscernible, then
\(R\) has a representation with support family $
 \mathcal S_0\subseteq\mathcal S\setminus\{E\}$. 
In particular,
\(\cov_V(\mathcal S_0)\geq\cov_V(\mathcal S)\).
\end{lemma}

\begin{proof}
Choose \(L<\omega\), \(S_\ell\in\mathcal S\setminus\{E\}\), tuples
\(\bar d_\ell^\epsilon=(d_{\ell,q}^\epsilon)_{q\in S_\ell}
\in P_{S_\ell}\) (\(\ell<L\), \(\epsilon<2\)), and a formula
\(\xi((z_\ell)_{\ell<L})\) such that the assignment
\(d_{\ell,q}^0\mapsto d_{\ell,q}^1\)
(\(\ell<L,\ q\in S_\ell\)) is a well-defined isomorphism between
the corresponding induced substructures after \(R\) is omitted,
while
\[
 \xi\bigl(
   (a_{S_\ell}(\bar d_\ell^0))_{\ell<L}
 \bigr)
 \quad\text{and}\quad
 \xi\bigl(
   (a_{S_\ell}(\bar d_\ell^1))_{\ell<L}
 \bigr)
\]
have opposite truth values.

Put \(W=\{d_{\ell,q}^0:\ell<L,\ q\in S_\ell\}\) and
\(d_{\ell,q}=d_{\ell,q}^0\).  For \(\epsilon<2\), let \(B^\epsilon\)
be the $\mathcal{L}_{V}$-structure on \(W\) for which
\(d_{\ell,q}\mapsto d_{\ell,q}^\epsilon\) is an isomorphism onto the
structure induced by $\mathcal{G}_{V}$ on
\(\{d_{\ell,q}^\epsilon:\ell<L,\ q\in S_\ell\}\).
Then \(B^0\) and \(B^1\) have the same parts and orders.  Write $
 R^{B^0}\mathbin{\triangle}R^{B^1}
 =\{\bar c_1,\ldots,\bar c_N\}$. 
For \(0\leq i\leq N\), let the $\mathcal{L}_{V}$-structure \(C_i\) have the same vertex set, parts,
and orders as \(B^0\), and let $
 R^{C_i}
 :=
 R^{B^0}\mathbin{\triangle}\{\bar c_j:1\leq j\leq i\}$. 
Then \(C_0=B^0\), \(C_N=B^1\), and $
 R^{C_{i-1}}\mathbin{\triangle}R^{C_i}=\{\bar c_i\}$ for all $i \in [N]$. 
Each \(C_i\) belongs to \(\mathcal K_V\), so choose an embedding
\(f_i:C_i\to\mathcal G_V\) and define \(\eta_i\in\{0,1\}\) by
\[
 \eta_i=1
 \quad\Longleftrightarrow\quad
 \mathbb M\models\xi\bigl(
   (a_{S_\ell}((f_i(d_{\ell,q}))_{q\in S_\ell}))_{\ell<L}
 \bigr).
\]
By \(V\)-indiscernibility, \(\eta_i\) is independent of the choice
of \(f_i\).  We may take \(f_0\) to be the inclusion and
\(f_N(d_{\ell,q})=d_{\ell,q}^1\).  Hence \(\eta_0\ne\eta_N\) by assumption, so
\(\eta_{i-1}\ne\eta_i\) for some \(i \in  [N]\). Put \(A^0 := C_{i-1}\), \(A^1 := C_i\), and
\(\bar c=(c_q)_{q\in V} :=\bar c_i\).  After interchanging \(A^0,A^1\) and possibly
negating \(\xi\), we may assume that, for every \(\epsilon<2\) and
every embedding \(f:A^\epsilon\to\mathcal G_V\),
\[
 A^\epsilon\models R(\bar c)
 \quad\Longleftrightarrow\quad
 \mathbb M\models\xi\bigl(
   (a_{S_\ell}((f(d_{\ell,q}))_{q\in S_\ell}))_{\ell<L}
 \bigr)
 \quad\Longleftrightarrow\quad
 \epsilon=1.
\]
Fix an embedding \(A^0\to\mathcal G_V\) and identify \(W\) with its
image in $\mathcal G_V$.

Apply Lemma~\ref{lem:cone} to \(W\) and \(\bar c\), and fix an
isomorphism \(g\) from \(\mathcal G_V\) onto the structure induced
on \(\bigcup_qI_q\).
For \(\ell<L\), let $
 T_\ell := \{q\in S_\ell:d_{\ell,q}=c_q\}$ 
and, for \(\bar y\in P_{T_\ell}\), define
\[
 \widetilde a_\ell(\bar y)
 := 
 a_{S_\ell}((e_{\ell,q})_{q\in S_\ell}),
 \qquad
 e_{\ell,q} := 
 \begin{cases}
  g(y_q),&q\in T_\ell,\\
  d_{\ell,q},&q\notin T_\ell.
 \end{cases}
\]

For an arbitrary \(\bar y\in P_V\), the map fixing
\(W\setminus\{c_q:q\in V\}\) and sending \(c_q\) to \(g(y_q)\) is an
embedding of \(A^0\) into \(\mathcal G_V\) if \(\mathcal G_V \models \neg R(\bar y)\), and
of \(A^1\) if \(\mathcal G_V \models R(\bar y)\), by Lemma~\ref{lem:cone}.  Hence
\begin{equation}\label{eq:removed-support-representation}
 \xi\bigl(
   (\widetilde a_\ell(\bar y|_{T_\ell}))_{\ell<L}
 \bigr)
 \quad\Longleftrightarrow\quad
 R(\bar y).
\end{equation}

Let \(\mathcal S_0 := \{S_\ell:\ell<L\}\).  For \(S\in\mathcal S_0\)
and \(\bar y\in P_S\), set
\[
 a'_S(\bar y)
 :=
 \bigl(
   \widetilde a_\ell(\bar y|_{T_\ell}):
   \ell<L,\ S_\ell=S
 \bigr).
\]
After grouping in \(\xi\) the variables with the same \(S_\ell\), \eqref{eq:removed-support-representation} says that this
\(\mathcal S_0\)-array represents \(R\).  Finally,
\(\mathcal S_0\subseteq\mathcal S\setminus\{E\}\), and so 
\(\cov_V(\mathcal S)\leq\cov_V(\mathcal S_0)\) is clear from the definition.
\end{proof}

The  following lemma is based on the proof in \cite[Section 5]{chernikov2021n}, but replaces type-counting with a direct shattering extraction.

\begin{lemma}[Compressing a support]\label{lem:compression}
Let \(a\) be an \(\mathcal S\)-array representing
\(R\) via \(\rho((z_S)_{S\in\mathcal S})\) and \(V\)-indiscernible over the parameters of $\rho$.  Let
\(E\in\mathcal S\) satisfy $
 2\leq|E|<|V|$, 
and suppose that the array \((a_S)_{S\in\mathcal S\setminus\{E\}}\) is
order-indiscernible.  Put $
 J :=V\setminus E$, $
 V' := \{\ast\}\mathbin{\dot\cup}J$, 
where \(\ast\) is first and \(J\) retains its order from \(V\).
Then the hyperedge relation \(R\) of \(\mathcal G_{V'}\) has a representation with support
family $
 \mathcal S'
 :=
 \{\{\ast\}\}\cup
 \{S\setminus E:S\in\mathcal S,\ S\nsubseteq E\}$,  
and $
 \cov_{V'}(\mathcal S')\geq\cov_V(\mathcal S)$. 
\end{lemma}

\begin{proof}
We let
\begin{gather*}
	\mathcal A :=\{S\in\mathcal S:S\subsetneq E\}, \   \mathcal B :=\{S\in\mathcal S:S\nsubseteq E\}, \   U_S :=S\setminus E \textrm{ for } S\in\mathcal B.
\end{gather*}
Choose \(q_0\in E\), a tuple
\(\bar d=(d_e)_{e\in E}\in P_E\), and strictly increasing sequences $
 (s_i^q:i \in \mathbb{N}_{>0})\subseteq P_q$ for $q\in J$. 
For \(S\in\mathcal A\), let \(e_S := a_S(\bar d|_S)\). For \(S\in\mathcal B\) and
\(\bar i\in\mathbb N_{>0}^{U_S}\), let 
\[
 r_q :=
 \begin{cases}
  d_q \textrm{ if }   q\in S\cap E,\\
  s_{\bar i(q)}^q \textrm{ if }  q\in U_S
 \end{cases}
 (q\in S),  \qquad 
 b_S(\bar i) := a_S((r_q)_{q\in S}).
\]

\noindent Fix \(m\geq1\) and \(X\subseteq[m]^J\).  The tuples $
 (s_{\bar i(q)}^q)_{q\in J}$ for $\bar i\in[m]^J$
are distinct.  Lemma~\ref{lem:one-point}, applied over $
 \{d_e:e\in E\}\cup
 \{s_i^q:q\in J,\ i\in[m]\}$ 
using the interval above \(d_{q_0}\), gives
\(t_X\in P_{q_0}\setminus\{d_{q_0}\}\) with the following property.
Letting \(\bar d^X=(d_q^X)_{q\in E}\in P_E\) be obtained from
\(\bar d\) by replacing \(d_{q_0}\) with \(t_X\), and for $\bar i\in[m]^J$ letting  
\[
 p_{\bar i,q}^X
 :=
 \begin{cases}
  d_q^X \textrm{ if }  q\in E,\\
  s_{\bar i(q)}^q \textrm{ if }   q\in J
 \end{cases} 
  \ (q\in V), 
 \qquad
 \bar p_{\bar i}^X := (p_{\bar i,q}^X)_{q\in V}\in P_V.
\]
we have that for every \(\bar i\in[m]^J\), $
\mathcal{G}_{V} \models  R(\bar p_{\bar i}^X) 
\ \Longleftrightarrow \ 
 \bar i\in X$. 

The map sending \(t_X\) to \(d_{q_0}\) and fixing every
\(d_q\), \(q\in E\setminus\{q_0\}\), and every
\(s_i^q\), \(q\in J\), \(i\in[m]\), is an isomorphism of the induced structures after
\(R\) is omitted.  By the definitions of \(e_S\) and \(b_S\), only
components \(a_S\) with \(S\ne E\) occur in the two tuples below.
Thus the assumption of order-indiscernibility of the array \((a_S)_{S\in\mathcal S\setminus\{E\}}\)  gives
\[
\left(
 (a_S(\bar d^X|_S))_{S\in\mathcal A},
 (a_S(\bar p_{\bar i}^X|_S))_{
    S\in\mathcal B,\ \bar i\in[m]^J}
\right)
\equiv
\left(
 (e_S)_{S\in\mathcal A},
 (b_S(\bar i|_{U_S}))_{
    S\in\mathcal B,\ \bar i\in[m]^J}
\right).
\]
Choose an automorphism \(\sigma_X\) of $\mathbb{M}$ fixing the 
parameters of $\rho$ and such that
\[
\begin{aligned}
 \sigma_X(a_S(\bar d^X|_S))&=e_S
 &&\textrm{for all } S\in\mathcal A,\\
 \sigma_X(a_S(\bar p_{\bar i}^X|_S))
 &=b_S(\bar i|_{U_S})
 &&\textrm{for all } S\in\mathcal B,\ \bar i\in[m]^J,
\end{aligned}
\]
and let \(x_X := \sigma_X(a_E(\bar d^X))\).
Let \(\theta(z_\ast;(z_S)_{S\in\mathcal B})\) be obtained from
\(\rho((z_S)_{S\in\mathcal S})\) by substituting \(e_S\) for \(z_S\)
when \(S\in\mathcal A\) and renaming \(z_E\) as \(z_\ast\). Since
\(\mathcal S=\{E\}\mathbin{\dot\cup}\mathcal A
\mathbin{\dot\cup}\mathcal B\) and
\(\bar p_{\bar i}^X|_E=\bar d^X\), while
\(\bar p_{\bar i}^X|_S=\bar d^X|_S\) for \(S\in\mathcal A\), the
definitions give, for every \(\bar i\in[m]^J\),
\begin{gather}
	\mathbb M\models
 \theta\bigl(
   x_X;(b_S(\bar i|_{U_S}))_{S\in\mathcal B}
 \bigr) \  \Leftrightarrow  \ 
 \mathbb M\models\rho\bigl(
   (\sigma_X(a_S(\bar p_{\bar i}^X|_S)))_{S\in\mathcal S} \label{eq: theta still reps}\\
   \  \Leftrightarrow \ 
 \mathbb M\models\rho\bigl(
   (a_S(\bar p_{\bar i}^X|_S))_{S\in\mathcal S}
 \bigr)
 \bigr) \Leftrightarrow  \ 
 \mathcal G_V\models R(\bar p_{\bar i}^X)
 \ \Leftrightarrow \ 
 \bar i\in X. \notag
\end{gather}
Here the four equivalences use, respectively, the definitions of
\(\theta,x_X,e_S,b_S\), the fact that \(\sigma_X\) fixes the parameters of $\rho$, the
representation of \(R\) by \(a\), and the choice of \(t_X\).

Choose bijections $
 j_q:P_q^{\mathcal G_{V'}}\longrightarrow\mathbb N_{>0}$ for $q\in J$. For \(p\in P_\ast^{\mathcal G_{V'}}\), let
\[
 X_p := 
 \left\{
  (j_q(v_q))_{q\in J}:
  (v_q)_{q\in J}\in
  \prod_{q\in J}P_q^{\mathcal G_{V'}},
  \ 
  \mathcal G_{V'}\models R(p,(v_q)_{q\in J})
 \right\}.
\]
Every finite subset of \(\mathbb N_{>0}^J\) is contained in
\([m]^J\) for some \(m\).  Hence \eqref{eq: theta still reps} and
compactness give, for each \(p\in P_\ast^{\mathcal G_{V'}}\), a
tuple \(c_\ast(p)\) such that
\[
 \models \theta\bigl(
   c_\ast(p);
   (b_S(\bar i|_{U_S}))_{S\in\mathcal B}
 \bigr)
 \quad\Longleftrightarrow\quad
 \bar i\in X_p
 \qquad\textrm{for all } \bar i\in\mathbb N_{>0}^J.
\]

\noindent Let $
 \mathcal U :=\{U_S:S\in\mathcal B\}$, $\mathcal S' := \{\{\ast\}\}\cup\mathcal U$. 
Define an \(\mathcal S'\)-array
\(a'=(a'_T)_{T\in\mathcal S'}\) by
\[
\begin{aligned}
 a'_{\{\ast\}}((p))
 &:=c_\ast(p)
 &&(p\in P_\ast^{\mathcal G_{V'}}),\\
 a'_U((v_q)_{q\in U})
 & :=
 \bigl(
  b_S((j_q(v_q))_{q\in U})
 \bigr)_{\substack{S\in\mathcal B\\U_S=U}}
 &&(U\in\mathcal U).
\end{aligned}
\]
For \(U\in\mathcal U\), let $
 w_U=(w_{U,S})_{\substack{S\in\mathcal B\\U_S=U}}$, 
where \(w_{U,S}\) has the same arity and sorts as \(z_S\).
Let \(w_{\{\ast\}}\) have the same arity and sorts as \(z_\ast\),
and define $
 \rho'\bigl((w_T)_{T\in\mathcal S'}\bigr)
 :=
 \theta\bigl(
   w_{\{\ast\}};
   (w_{U_S,S})_{S\in\mathcal B}
 \bigr)$. 
For
\(\bar v=(v_q)_{q\in V'}\in P_{V'}^{\mathcal G_{V'}}\), put
\(\bar i := (j_q(v_q))_{q\in J}\).  Then, for all $\bar{v} \in P_{V'}^{\mathcal{G}_{V'}}$, 
\begin{gather*}
	\mathbb M\models
 \rho'\bigl(
   (a'_T(\bar v|_T))_{T\in\mathcal S'}
 \bigr) \  \Leftrightarrow  \ 
 \mathbb M\models
 \theta\bigl(
   c_\ast(v_\ast);
   (b_S(\bar i|_{U_S}))_{S\in\mathcal B}
 \bigr) 
 \Leftrightarrow \ 
 \bar i\in X_{v_\ast}
 \ \Leftrightarrow \ 
 \mathcal G_{V'}\models R(\bar v).
\end{gather*}
Thus \((a',\rho')\) is a representation of \(R\) with support
family \(\mathcal S'\).

Finally, let \(\mathcal F'\subseteq\mathcal S'\) attain
\(\cov_{V'}(\mathcal S')\).  The only member of \(\mathcal S'\)
containing \(\ast\) is \(\{\ast\}\).  If
\(\{\ast\}\notin\mathcal F'\), adjoining it increases
\(|\mathcal F'|\) by one and decreases
\(\left|V'\setminus\bigcup\mathcal F'\right|\) by one.
We may therefore assume that \(\{\ast\}\in\mathcal F'\). For each \(U\in\mathcal F'\setminus\{\{\ast\}\}\), choose
\(S_U\in\mathcal B\) with \(S_U\setminus E=U\), and put $
 \mathcal F
 :=
 \{E\}\cup
 \{S_U:U\in\mathcal F'\setminus\{\{\ast\}\}\}
 \subseteq\mathcal S$. 
Distinct \(U\)'s give distinct \(S_U\)'s, and no \(S_U\) equals
\(E\).  Hence $
 |\mathcal F|=|\mathcal F'|$, 
 $V\setminus\bigcup\mathcal F
 =
 J\setminus\bigcup
   (\mathcal F'\setminus\{\{\ast\}\})
 =
 V'\setminus\bigcup\mathcal F'$. 
Hence $
 \cov_V(\mathcal S)
 \leq
 |\mathcal F|+
 \left|V\setminus\bigcup\mathcal F\right|
 =
 \cov_{V'}(\mathcal S')$. 
\end{proof}

Combining these lemmas, we deduce the key bound on the supports of arrays representing $R$:
\begin{prop}\label{prop:array-bound}
Suppose that an \(\mathcal S\)-array represents \(R\).  Let $
 r :=\cov_V(\mathcal S)-1$. 
If \(r\geq1\), then \(T\) has \(\IP_r\).
\end{prop}

\begin{proof}
We use lexicographic induction on
\((|V|,|\mathcal S|)\).  By
Lemma~\ref{lem:indiscernibilization}, we may assume that the array is
\(V\)-indiscernible.

First, \(\bigcup\mathcal S=V\).  Otherwise, for some
\(q\notin\bigcup\mathcal S\), fix the other coordinates and apply
Lemma~\ref{lem:one-point} twice with opposite incidences.  This gives
\(\bar p,\bar p'\in P_V\) that agree outside \(q\) and satisfy
\(R(\bar p)\leftrightarrow\neg R(\bar p')\). But the restrictions of $\bar p,\bar p'$ to
every support are equal, contradicting \eqref{eq:represents}.

If every support is a singleton, then
\(\mathcal S=\{\{q\}:q\in V\}\) and \(r=|V|-1\).  With any one
singleton block taken as the object variables, the representing
formula has \(\IP_r\): choose finite rows in the other parts and use
Lemma~\ref{lem:one-point} to realize each subset of their Cartesian
product.  This proves the proposition in this case.

So we can choose \(E\in\mathcal S\) with \(|E|\geq2\).  Since
\(\cov_V(\mathcal S)=r+1\geq2\), we have \(E\ne V\).
If \((a_S)_{S\in\mathcal S\setminus\{E\}}\) is not
order-indiscernible, Lemma~\ref{lem:remove-support} gives a
representation with support family $
 \mathcal S_0\subseteq\mathcal S\setminus\{E\}$, $
 r_0 := \cov_V(\mathcal S_0)-1\geq r$. 
The induction hypothesis gives \(\IP_{r_0}\), hence \(\IP_r\) by
Remark~\ref{rem:basic}.

Otherwise Lemma~\ref{lem:compression} gives a representation on $
 V'=\{\ast\}\mathbin{\dot\cup}(V\setminus E)$, $|V'|=|V|-|E|+1<|V|$ 
with \(r':=\cov_{V'}(\mathcal S')-1\geq r\).  The induction
hypothesis gives \(\IP_{r'}\), and hence \(\IP_r\).
\end{proof}

\section{The Composition Lemma}

\begin{theorem}\label{thm:composition}
Let \(d,k\geq1\), let \(\mathcal L\subseteq\mathcal L'\), and let
\(\mathcal M\) be an \(\mathcal L'\)-structure.  Let \(I\) be finite
and let \(\rho((x_\alpha)_{\alpha\in I})\) be an
\(\mathcal L\)-formula.  For each \(\alpha\in I\), let $
 S_\alpha\subseteq\{0,\ldots,d\}$ 
and let \(t_\alpha((y_s)_{s\in S_\alpha})\) be a
tuple-valued \(\mathcal L'\)-definable map whose values have the
same sorts as \(x_\alpha\).
Let 
\[
 \psi(y_0;y_1,\ldots,y_d)
 :=
 \rho\bigl(
   (t_\alpha((y_s)_{s\in S_\alpha}))_{\alpha\in I}
 \bigr).
\]
Taking $\mathcal S := \{S_\alpha:S_\alpha\ne\varnothing\}$ and  
 $c := \cov_{\{0,\ldots,d\}}(\mathcal S)$, if \(c\geq2\) and \(\psi\) has \(\IP_d\), then the theory of the
\(\mathcal L\)-reduct of \(\mathcal M\) has \(\IP_{c-1}\).

 In
particular, if $n \geq 1$, \(|S_\alpha|\leq k\) for every \(\alpha \in I\), $d = nk$ and  the $\mathcal{L}$-reduct is
\(n\)-dependent, then  \(\psi\) is $nk$-dependent.
\end{theorem}

\begin{proof}
Put \(V=\{0,\ldots,d\}\). If \(\psi\) has \(\IP_d\), by
\cite[Proposition~5.2]{CPT}, after passing to an elementary
extension there are tuples \(b_p\), \(p\in\bigcup_{s\in V}P_s\),
where \(b_p\) has the sorts of \(y_s\) for \(p\in P_s\), such that $
\models  \psi(b_{p_0};b_{p_1},\ldots,b_{p_d})
 \ \Leftrightarrow \ 
 \mathcal G_V\models R((p_s)_{s\in V})$ for all $(p_s)_{s\in V}\in P_V$. 
For \(\alpha\in I\), let $
 a_\alpha((p_s)_{s\in S_\alpha})
 :=
 t_\alpha((b_{p_s})_{s\in S_\alpha})$. 
Then $
 \models \rho\bigl(
 (a_\alpha(\bar p|_{S_\alpha}))_{\alpha\in I}
 \bigr) \ \Leftrightarrow \ 
 \mathcal G_V\models R(\bar p)$ for all $\bar p\in P_V$. For \(S\in\mathcal S\) and \(\bar p\in P_S\), define $
 a_S(\bar p)
 :=
 (a_\alpha(\bar p))_{\alpha\in I,\ S_\alpha=S}$. 
After substituting in \(\rho\) the constant values from empty
supports and grouping its remaining variables accordingly, this
\(\mathcal S\)-array in the \(\mathcal L\)-reduct represents \(R\).
Hence
Proposition~\ref{prop:array-bound} gives \(\IP_{c-1}\).

If every set in $\mathcal{S}$ has size at most \(k\), then for every
\(\mathcal F\subseteq\mathcal S\), its members, together with the
uncovered coordinates as singletons, cover \(V\) by sets of size at
most \(k\).  Hence $
 d+1
 \leq
 k\left(
   |\mathcal F|+
   \left|V\setminus\bigcup\mathcal F\right|
 \right)$. 
Taking the minimum over \(\mathcal F\) gives $
 c\geq\left\lceil\frac{d+1}{k}\right\rceil$. 
Thus the reduct has
\(\IP_{\lceil(d+1)/k\rceil-1}\) whenever this index is positive, by
Remark~\ref{rem:basic}.  If \(d=nk\), then $
 \left\lceil\frac{nk+1}{k}\right\rceil-1=n$, 
contradicting \(n\)-dependence.
\end{proof}

\section{Optimality}\label{sec: Optimality}

\subsection{Dependence of the reduct}
The proof of Theorem~\ref{thm:composition} uses \(n\)-dependence of
the entire \(\mathcal L\)-reduct theory, rather than only
\(n\)-dependence of the formula \(\rho\): Lemma~\ref{lem:remove-support} uses an
arbitrary separating formula \(\xi\), which need not be an instance
or a Boolean combination of \(\rho\). In this section we demonstrate that simply assuming that the formula  $ \rho$ is $n$-dependent under arbitrary partitions of its variables is insufficient, already in the NIP case $n=1$ --- this gives a
negative answer to \cite[Problem~3.26(2)]{CHIII}. In fact, we obtain a more precise description of how much dependence of the base theory is needed.

We specialize to the setting of
\cite[Theorem~3.24]{CHIII}; namely, in the notation of Theorem~\ref{thm:composition}  we consider an $\mathcal{L}'$-formula  \begin{equation}\label{eq:CH-normal-form}
 \psi(y_0;y_1,\ldots,y_d)
 =
 \rho\bigl(
  t_0((y_j)_{j\ne0}),\ldots,
  t_d((y_j)_{j\ne d})
 \bigr),
\end{equation}
where \(\rho(x_0,\ldots,x_d)\) is an $\mathcal{L}$-formula and $t_\alpha$ are $d$-ary  \(\mathcal L'\)-definable maps on appropriate sorts. For \(i\leq d\), \(m\geq1\) and \(a\in[m]\), let \(x_i^a\)  be a 
tuple of variables of the same sorts as \(x_i\) (possibly with repetitions).   If
\(B(u_1,\ldots,u_m)\) is a propositional Boolean formula 
%and
%\[
% u_a:=\rho(x_0^a,\ldots,x_d^a)\qquad(a\in[m]), 
%\]
and \(\bar z\) contains every variable occurring in 
\((x_1^a : a \in [m])\) (and possibly any other variables), consider the $\mathcal{L}$-formula 
\begin{equation}\label{eq:CH-existential-formulas}
 \exists\bar z\,
 B\bigl(
   \rho(x_0^1,\ldots,x_d^1),\ldots,
   \rho(x_0^m,\ldots,x_d^m)
 \bigr).
\end{equation}

\begin{prop}\label{prop:CH-local-hypotheses}
Let \(\mathcal M\) be an $\mathcal{L}'$-expansion of an \(\mathcal L\)-structure
by the maps in \eqref{eq:CH-normal-form}.  Assume that, in the theory
of its \(\mathcal L\)-reduct, $
 \rho(x_1;x_0,x_2,\ldots,x_d)$ 
is NIP and every formula \eqref{eq:CH-existential-formulas} is NIP
under every partition of its free variables into two blocks (for all $m$ and $B$).  Then
\(\psi\) is \(d\)-dependent.
\end{prop}

\begin{proof}
This follows from the induction in the proof of
\cite[Lemma~3.23]{CHIII}, with the successive groupings of variables
used in the proof of \cite[Theorem~3.24]{CHIII}.  Its base case is
the ordinary Sauer--Shelah bound for
\(\rho(x_1;x_0,x_2,\ldots,x_d)\).  At each later stage, the only use
of NIP is for one of the formulas denoted
\(\xi_{\varphi,n,\eta}\) there.  After undoing the grouping, it is an
existentially quantified conjunction of instances of \(\rho\) and
their negations in which all coordinate-\(1\) variables, and possibly
variables from other coordinates, are quantified.  Thus it is an
instance of \eqref{eq:CH-existential-formulas}.  The other steps of
the induction use no further NIP assumption, and its conclusion is
precisely the hypothesis used in \cite[Theorem~3.24]{CHIII}.
\end{proof}

We next show both that stability (hence also NIP) of \(\rho\) under every bipartition of its variables 
does not suffice, and that the number \(m\) in
\eqref{eq:CH-existential-formulas} cannot be bounded uniformly (part (3) explains why
Proposition~\ref{prop:CH-local-hypotheses} fixes the quantified
coordinates \((x_1^a : a \in [m])\)). 
%In particular, parts (1) and (4) for \(N=1\) give a
%negative answer to \cite[Problem~3.26(2)]{CHIII}.

\begin{prop}\label{prop:local-hypotheses-sharp}
For every \(d\geq2\) and \(N\geq1\), there are languages
\(\mathcal L\subseteq\mathcal L'\), an \(\mathcal L'\)-structure
\(\mathcal M_N\), an \(\mathcal L\)-formula
\(\rho(x_0,\ldots,x_d)\), and \(\mathcal L'\)-definable maps $
 t_i:\prod_{j\ne i}D_j\to D_i$ for $i\leq d$, 
where \(x_i\)  have sort \(D_i\), with the following
properties.  
\begin{enumerate}
\item For every \(\emptyset \neq J\subsetneq\{0,\ldots,d\}\), the
      formula $
       \rho((x_i)_{i\in J};(x_i)_{i\notin J})$ 
      is stable. 
\item Every formula \eqref{eq:CH-existential-formulas} with
      \(m\leq N\) is stable under every partition of its free
      variables.
\item The formula $
       \exists x_2\cdots\exists x_d\,
       \rho(x_0,x_1,\ldots,x_d)$ 
      has the independence property.
\item The formula $
       \psi_N(y_0;y_1,\ldots,y_d)
       :=\rho\bigl(
          t_0((y_j)_{j\ne0}),\ldots,
          t_d((y_j)_{j\ne d})
        \bigr)$ 
      has \(\IP_d\) (where \(y_i\)  have sort \(D_i\)).
\end{enumerate}
\end{prop}

\begin{proof}
Put \(V := \{0,\ldots,d\}\).  Let \(X_i\), for \(i\in V\), be
countably infinite sets, and let $
 \Gamma\subseteq\prod_{i\in V}X_i$ 
be the edge relation of the countable generic
\((d+1)\)-partite \((d+1)\)-uniform hypergraph.  For \(i\in V\), put $
 Y_i:=\prod_{j\ne i}X_j$, 
and define $
 F:\prod_{i\in V}X_i\longrightarrow\prod_{i\in V}Y_i$ via  $
 F((a_i)_{i\in V}) := ((a_j)_{j\ne i})_{i\in V}$. 
Let
\[
 W:=\prod_{i\ne1}Y_i,
 \qquad
 Z_N:=[W]^{\leq N}=\{A\subseteq W:|A|\leq N\},
\]
where \(Z_N\) is taken disjoint from \(Y_1\), and set
\[
 D_1:=Y_1\mathbin{\dot\cup}Z_N,
 \qquad
 D_i:=Y_i\quad(i\ne1).
\]
Let \(\mathcal L\) be the many-sorted language with sorts \(D_i\), a
unary predicate \(Z\) on \(D_1\) naming \(Z_N\), and the relation
\(\rho\) on $\prod_{i\in V}D_i$ interpreted by
\begin{equation}\label{eq:bounded-membership}
 \rho(p_0,\ldots,p_d)
 \quad :\Longleftrightarrow\quad
 \begin{cases}
  (p_i)_{i\in V}\in F(\Gamma) &\textrm{ when } p_1\in D_1 \setminus Z_N,\\
 (p_i)_{i\ne1}\in p_1 &\textrm{ when } p_1\in Z_N.
 \end{cases}
\end{equation}
This defines an \(\mathcal L\)-structure, and we specify its $\mathcal{L}'$-expansion \(\mathcal M_N\) below.

First consider the two definable pieces of \(\rho\) given by
\(\neg Z(p_1)\) and \(Z(p_1)\).  Any two coordinates of a tuple in
the image of \(F\) determine its preimage.  Since \(d\geq2\), in
every bipartition of \(V\) one side contains at least two coordinates,
and fixing that side leaves at most one completion in \(\neg Z(p_1)\).  In \(Z(p_1)\), fixing the side containing coordinate
\(1\) leaves at most \(N\) completions.  A formula with uniformly
bounded fibers on one side is stable, and stable formulas are closed
under Boolean combinations.  This proves (1).

For (2), put \(B\) in full disjunctive normal form. As Boolean combinations preserve stability, it is enough to
consider a conjunction
\[
 \bigwedge_{a\in[m]}
 \rho(x_0^a,\ldots,x_d^a)^{\varepsilon_a},
 \qquad
 \varepsilon_a\in\{0,1\},
\]
where \(\varphi^1=\varphi\) and
\(\varphi^0=\neg\varphi\).  Fix a variable \(z\) of sort \(D_1\)
occurring among the \(x_1^a\), let
\(I_z := \{a\in[m]:x_1^a=z\}\), and put
\(w_a :=(x_i^a)_{i\ne1}\).  Then 
\begin{equation}\label{eq:eliminate-coordinate}
 \exists z\bigwedge_{a\in I_z}
 \rho(x_0^a,z,x_2^a,\ldots,x_d^a)^{\varepsilon_a}
 \quad\Longleftrightarrow\quad
 \bigwedge_{\substack{a,b\in I_z\\
                       \varepsilon_a\ne\varepsilon_b}}
 w_a\ne w_b.
\end{equation}
The right-hand side is clearly necessary.  Conversely, if it holds, then $
 c^{\ast} := \{w_a:a\in I_z\text{ and }\varepsilon_a=1\}\in Z_N$ 
is a witness, since \(|I_z|\leq m\leq N\).  Applying
\eqref{eq:eliminate-coordinate} to every variable occurring in
coordinate \(1\) reduces the formula to one in pure equality, with
possibly some variables from the other coordinates still
existentially quantified.  Such a formula is stable under every
partition.  

For (3), fix \(c_i\in X_i\) for \(2\leq i\leq d\).  For
\(a\in X_1\) and \(b\in X_0\), put
\[
 p_0(a):=(a,c_2,\ldots,c_d)\in Y_0,
 \qquad
 p_1(b):=(b,c_2,\ldots,c_d)\in Y_1.
\]
Then
\[
 \exists p_2\cdots\exists p_d\,
 \rho(p_0(a),p_1(b),p_2,\ldots,p_d)
 \quad\Longleftrightarrow\quad
 \Gamma(b,a,c_2,\ldots,c_d).
\]
The relation on the right is the countable generic bipartite graph,
so the formula in (3) has the independence property.

Finally, as each \(D_i\) is countably infinite, choose bijections
\(e_i:D_i\to X_i\), let \(\mathcal L'\) add function symbols \(t_i\),
and expand by setting $
 t_i((p_j)_{j\ne i}):=(e_j(p_j))_{j\ne i}\in Y_i\subseteq D_i$. 
Then $
 (t_i((p_j)_{j\ne i}))_{i\in V}
 =F((e_i(p_i))_{i\in V})$, 
and hence $
 \models \psi_N(p_0;p_1,\ldots,p_d) \Leftrightarrow 
 \Gamma(e_0(p_0),\ldots,e_d(p_d))$. 
Thus \(\psi_N\) has \(\IP_d\), proving (4).
\end{proof}

\subsection{Sharpness of the numerical bounds}

\begin{prop}\label{prop:array-sharp}
The bound in Proposition~\ref{prop:array-bound} is sharp.
\end{prop}

\begin{proof}
Let \(Q\) be a nonempty finite set, and let
\(\mathcal E=\{E_1,\ldots,E_s\}\) be a family of nonempty subsets of
\(Q\) such that $
 Q=\bigcup_{j=1}^sE_j$, $
 \cov_Q(\mathcal E)=s$.  We show that there is an \((s+1)\)-dependent theory in which the generic
\(\{\ast\}\mathbin{\dot\cup}Q\)-partite hypergraph has a
representation with support family $
 \{\{\ast\},E_1,\ldots,E_s\}$.

Let \(H(x;z_1,\ldots,z_s)\) be the edge relation of the unordered
generic \((s+1)\)-partite \((s+1)\)-uniform hypergraph, with parts
\(Z_0,\ldots,Z_s\).  Its theory is \((s+1)\)-dependent and \(H\) has
\(\IP_s\), by \cite[Example~2.2(2)]{CPT}. Work in a sufficiently saturated model of this theory.  Order \(Q\)
arbitrarily and put \(V=\{\ast\}\mathbin{\dot\cup}Q\), with
\(\ast\) first.  For \(1\leq j\leq s\), choose an injection $
 b_j:P_{E_j}\longrightarrow Z_j$. 
Because the \(E_j\)'s cover \(Q\), the tuples $
 (b_j(\bar p|_{E_j}))_{1\leq j\leq s}$ 
are pairwise distinct for $\bar p\in P_Q$.  For each \(u\in P_\ast\), the extension
property and saturation give \(b_\ast(u)\in Z_0\) such that $
 H\bigl(
   b_\ast(u);
   b_1(\bar p|_{E_1}),\ldots,b_s(\bar p|_{E_s})
 \bigr)
 \ \Leftrightarrow \ 
 \mathcal G_V\models R(u,\bar p)$ for all $\bar p\in P_Q$. 
These maps give the required representation.
The covering number $\cov_{V}(
 \{\{\ast\},E_1,\ldots,E_s\})$ of the support of this representation  is
\(1+\cov_Q(\mathcal E)=s+1\), while the ambient theory has
\(\IP_s\) but not \(\IP_{s+1}\).
\end{proof}

\begin{prop}\label{prop:composition-sharp}
For all \(n,k\geq1\) with \(nk\geq2\), the bound \(nk\) in
Theorem~\ref{thm:composition} cannot be lowered. 
\end{prop}

\begin{proof}
Fix a partition $
 \{0,\ldots,nk-1\}
 =
 B_0\mathbin{\dot\cup}\cdots\mathbin{\dot\cup}B_{n-1}$ with $ |B_j|=k$ and $ 0\in B_0$. 
 Put \(J=\{1,\ldots,nk-1\}\) and
\(C=\bigcup_{1\leq j<n}B_j\).
Then \(J=(B_0\setminus\{0\})\mathbin{\dot\cup}C\).
Let \(\mathcal H\) be a sufficiently saturated unordered generic
\(n\)-partite \(n\)-uniform hypergraph, with parts
\(Z_0,\ldots,Z_{n-1}\) and edge relation \(H\).  For \(n=1\), this
means an infinite set with an infinite and coinfinite unary
relation.

For \(s\in J\), let \(Y_s :=\{c_i^s:i<\omega\}\) with \(c_i^s\) distinct, and let $
 Y_0 := \mathcal P(\omega^J)$.  The theory of the
\(\mathcal L\)-structure consisting of \(\mathcal H\) and the
disjoint sorts \(Y_s\) (with no additional structure) eliminates quantifiers and is
\(n\)-dependent by
\cite[Example~2.2(3)]{CPT} when \(n\geq2\); for \(n=1\), it is
stable. For \(1\leq j<n\), choose an injection $
 f_j:\prod_{s\in B_j}Y_s\longrightarrow Z_j$. 
For \(\eta\in\omega^C\), let $
 q_{\eta,j}
 :=
 f_j((c_{\eta(s)}^s)_{s\in B_j})$ for $1\leq j<n$. 
The tuples \((q_{\eta,j})_{1\leq j<n}\) are pairwise distinct.
For \(A \in Y_0\) and
\(\xi\in\omega^{B_0\setminus\{0\}}\), the extension and  saturation
 give \(p_{A,\xi}\in Z_0\) such that
\[
\models  H(p_{A,\xi},q_{\eta,1},\ldots,q_{\eta,n-1})
 \quad\Longleftrightarrow\quad
 \xi\cup\eta\in A
 \qquad \textrm{ for all } \eta\in\omega^C.
\]
When \(n=1\), this reads
\(H(p_{A,\xi}) \Leftrightarrow \xi\in A\).
Define $
 f_0:
 Y_0\times\prod_{s\in B_0\setminus\{0\}}Y_s
 \longrightarrow Z_0$
by $
 f_0\bigl(
   A,(c_{\xi(s)}^s)_{s\in B_0\setminus\{0\}}
 \bigr)
 :=
 p_{A,\xi}$, and let
\[
 \theta(y_0;y_1,\ldots,y_{nk-1})
 :=
 H\bigl(
  f_0((y_s)_{s\in B_0}),\ldots,
  f_{n-1}((y_s)_{s\in B_{n-1}})
 \bigr).
\]
For \(A\subseteq\omega^J\) and \(\iota\in\omega^J\), 
 $\theta\bigl(A;(c_{\iota(s)}^s)_{s\in J}\bigr) \ \Leftrightarrow \  \iota\in A$. 
Thus \(\theta\) has \(\IP_{nk-1}\).
\end{proof}
\bibliographystyle{plain}
\bibliography{CompositionLemma}

\end{document}